\documentclass[12pt]{amsart}

\usepackage{graphicx}

\usepackage{subcaption} 

\usepackage{tikz-cd}
\usepackage{hyperref}
\hypersetup{bookmarks=true,
	unicode=true,
	colorlinks=true,
	citecolor=black,
	linkcolor=black,
	urlcolor=black,
	plainpages=false,
	pdfpagelabels=true}

\usepackage{url}	
\allowdisplaybreaks 

\usepackage{pgf}

\usepackage{xcolor}

\usepackage{comment} 

\usepackage{tikz}
\usetikzlibrary{arrows}
\usetikzlibrary{decorations.markings,arrows,automata,arrows,backgrounds,snakes}
\usepackage{tikz-cd} 
\usepackage{pgf}
\usetikzlibrary{babel}

\usepackage{appendix}

\usepackage{mathrsfs}
\usepackage{dsfont}

\usepackage{mathtools}

\usepackage{accents}

\usepackage[shortlabels]{enumitem}
\usepackage{txfonts}

\newtheorem{theorem}{Theorem}[section]

\theoremstyle{definition}

\theoremstyle{remark}

\newtheorem{question}[theorem]{Question}

\usepackage{color}
\definecolor{darkgreen}{cmyk}{1,0,1,.2}
\definecolor{m}{rgb}{1,0.1,1}

\newdimen\theight
\def\TeXref#1{%
	\leavevmode\vadjust{\setbox0=\hbox{{\tt
				\quad\quad  {\small \textrm #1}}}%
		\theight=\ht0
		\advance\theight by \lineskip
		\kern -\theight \vbox to
		\theight{\rightline{\rlap{\box0}}%
			\vss}%
}}%


\begin{document}

\thanks{The author was partially supported by the grant ED431C 2023/31 (Xunta de Galicia, FEDER) and by Programa de axudas á etapa predoutoral da Xunta de Galicia.
 }
    
	\title{On Lefschetz's point-free periodicity}

	\author[A. Majadas-Moure 
	]{%
		Alejandro O. Majadas-Moure  
	}

	\address{
		Alejandro O. Majadas-Moure \\
		Departamento de Matemáticas, Universidade de Santiago de Compostela, Spain}
	\email{alejandro.majadas@usc.es}

	\begin{abstract} 
		We motivate two new approaches to the study of Lefschetz point-free periodicity. The first focusses on spaces satisfying the Wecken property. As an example, we study the bubble spaces. The second is a relative study of the Lefschetz point-free periodicity. This becomes important, for example, in the study of repellers of dynamical systems.
	\end{abstract}
	
	
	
	\maketitle
	\section{Introduction}
    It is an important question in dynamical systems to know when a map $f:X\rightarrow X$ has no periodic points. If $X$ is a compact simplicial complex or a compact ANR (we can consider, however, compact manifolds), a necessary condition for this is that all Lefschetz numbers of the iterates of $f$ vanish, that is, $L(f^m,X)=0$ for all $m\in\mathbb{Z}^+$. If for some $m$ we have $L(f^m,X)\neq 0$, then, the Lefschetz fixed point theorem implies that $f^m$ has a fixed point and then $f$ must have a periodic point.

    Many examples in the literature focus on the study of sufficient and/or necessary conditions to have $L(f^m,X)=0$ for each $m$, for example \cite{G-L, G-L2, G-L3, G-S-U, G-K-N-S, H-R, L, L-S1, L-S2, S2, S1}. 
    
    The relevance of this note lies not so much in the theorems as in justifying the importance of new approaches to the study of periodicity. The theorems we present here may be considered as the first examples of this new study, but many other results can derived along the same line.

    The new approaches that we would like to highlight follow from two different questions:
    
    First, in many cases in previous publications (see for example \cite{G-L, G-L2, G-L3, G-S-U, G-K-N-S, H-R, L, L-S1, L-S2, S2, S1}), the spaces that were studied didn't satisfy the Wecken property (this means that the Nielsen number of a map is the least number of fixed points in the homotopy class of the map), so we cannot even hope that $L(f,X)= 0$ implies that $f$ is fixed-point free. In particular, the study of bouquets of different spaces (for instance, spheres) has been intensively studied \cite{G-S-U, L-S1, L-S2, S1}, even if these spaces are very bad-behaved with the Wecken property since they have a global separating point. That is why, in Section~\ref{sec burbuja}, we suggest the use of bubble spaces, a kind of variation of the bouquets of spheres that is well-behaved with the Wecken property. For dimension greater than two, they satisfy this property and, above dimension $5$ it is more reasonable that they would satisfy that the \textit{Nielsen type number} for the $n$-iterate, $NF_n(f)$ (see \cite[Chapter 3]{Jiang}) equals $\mathrm{Min}\{\#\mathrm{Fix}(g^n)\,|\,g\simeq f\}$ \cite[Chapter III, Theorem 4.14]{Jiang}.

    In the setting of bubble spaces, we prove two theorems concerning necessary and sufficient conditions for $L(f^m,X)=0$ for all $m$. Theorem~\ref{thm primera generalizacion} is a generalization of \cite[Theorem 2.4 (c)]{S1} to bubble spaces, since its proof can be reduced to the bouquet case. Moreover, Theorem~\ref{thm respuesta a pregunta abta} is a more powerful result where we generalize the case of a finite bouquet of arbitrary spheres. This is closed related to Question (1) raised in \cite{S1}.

    In Section~\ref{seccion relativo}, we focus on the following question:

Given $f:(X,A)\rightarrow (X,A)$, are there necessary and/or sufficient conditions to ensure that $f$ has no periodic points in $\overline{X\setminus A}$? This is the second point of our new approach, since we also find it very interesting not only to know when a map is periodic free, but also when it is periodic free outside certain subspace. Among other applications, for example, this turns out to be useful when we want to know if a space is a repeller in a dynamical system.

Just in order to illustrate this new approach to the study of periodicity, in Section~\ref{seccion relativo} we present some examples that show a way to obtain some new results with this relative viewpoint. We start by generalizing some theorems in \cite{G-K-N-S}. Then we give a very illustrative necessary condition for a map in the sphere to have only periodic points in a ring neighborhood of the equator and, finally, in Theorem~\ref{thm ultimo teor}, we give an example of the kind of theorems we can obtain using a pair that has the relative Wecken property \cite[Theorem 6.7]{Zhao}. We insist that many other examples can follow using similar techniques.


    
In the whole paper, $X$ will be a connected simplicial complex. We will also use the Lefschetz zeta function:

\begin{equation*}
    \zeta_f(t):=\mathrm{exp}\left(\sum_{m\geq 1}\frac{L(f^m)}{m}t^m\right).
\end{equation*}
With the conditions of Section~\ref{sec burbuja}, it is proved in \cite[Proposition 5.13]{F} that this expression is equal to:
\begin{equation}\label{ecuacion zeta lefschetz}
    \zeta_f(t)=\prod_{k=0}^{n}\mathrm{det}(\mathrm{Id}_k-tf_{\ast k})^{(-1)^{k+1}}.
\end{equation}
\section{Bubble spaces}\label{sec burbuja}
In \cite{S1}, a detailled study of the property $L(f^m)=0$ is performed in bouquets of spheres (of the same or of different dimensions). The advantage of the sums in $\mathrm{TOP}^0$ (the category of pointed spaces) is that their (co)homology groups are very easy to compute from the (co)homology groups of the summands. However, the gluing point of the bouquet represents a notorious obstacle in Nielsen theory: it is a global separation point, and hence this poses a serious problem in terms of the Wecken property. A possible solution to this problem involves another way to glue spheres that turns out to be more appropriate in Nielsen theory: the \textit{bubble spaces}. The double-bubble space is a very important concept in geometric analysis. In general, in this note we will define a \textit{bubble space} as a recursive gluing of disks by injective maps of their boundary.

Formally, we start with a sphere $S^n$ (in order to avoid trivial cases we take $n\geq 2$). Next, we glue a disk $D^{n_1}$, $1\leq n_1\leq n$ through an imbedding $f_{n_1}:\partial D^{n_1}\rightarrow S^n$. Hence, the resulting space of this second step is the pushout of the following diagram:
\begin{equation*}
    D^{n_1}\hookleftarrow S^{n_1-1}\overset{f_{n_1}}{\rightarrow}S^n.
\end{equation*}
For the third step, we glue another disk $D^{n_2}$ either to $S^n\setminus \mathrm{Im}(f_{n_1})$ or to the interior of the disk $D^{n_1}$ (in this case, we will also ask $n_2\leq n_1$). The result of repeating this process a finite number of times is what we will call a bubble space. When considering the standard inclusion $S^n\subset\mathbb R^{n+1}$, we allow bubbles inside bubbles (that is, we can glue them both in the convex and the concave sides of the previous bubbles). Figure~\ref{espacio burbuja} illustrates this kind of spaces.

\begin{figure}[htb] 
   \centering
    \includegraphics[scale=0.35]{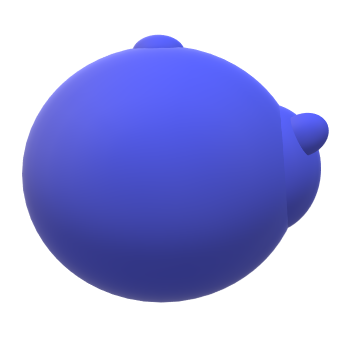} 
    \caption{A bubble space.}
   \label{espacio burbuja}
\end{figure}

Each bubble space $X$ will then consist in a sphere $S^n$ and a series of disks
\begin{equation*}
    D^{i_1}_1,\ldots,D^{i_1}_{k_{i_1}},\ldots, D^{i_m}_1,\ldots, D^{i_m}_{k_{i_m}},\;\text{with}\; i_1>\ldots>i_m>0.
\end{equation*}
Independently of the disk where each disk in glued to, an application of the Mayer-Vietoris theorem leads to the following homology groups with rational coefficients (the same idea works using cohomology):
        \[ H_l(X) = \begin{cases} 
          \mathbb{Q} & \text{if}\;l=n \\
          \mathbb{Q}^{k_{i_j}} & \text{if}\;l=i_j \\
          \mathbb{Q} & \text{if}\, l=0\\
          0 & \text{elsewhere}
       \end{cases}
    \]
if $n>i_1$, and
 \[ H_l(X) = \begin{cases} 
          \mathbb{Q}^{k_{i_1}+1} & \text{if}\;l=n, \\
          \mathbb{Q}^{k_{i_j}} & \text{if}\;l=i_j \;\text{and}\; j>1, \\
          \mathbb{Q} & \text{if}\, l=0,\\
          0 & \text{elsewhere}
       \end{cases}
    \]
    if $n=i_1$.
Since the homology (and also the homotopy class) is the same as in a bouquet of spheres, it is not very surprising that we can recover some of the results in \cite{S1}. Let us see, for example, a generalization of \cite[Theorem 2.4 (c)]{S1} with a different proof.
\begin{theorem}\label{thm primera generalizacion}
    Suppose that our bubble space $X$ consists of a sphere $S^{n_r}$ and $r-1$ disks $D^{n_1},\ldots,D^{n_{r-1}}$, with $n_1<,\ldots,<n_r$. Let $E=\{i\mid n_i \;\text{even}\}$ and $O=\{i\mid n_i \;\text{odd}\}$. Then $L(f^m,X)=0$ for all $m\in\mathbb{N^+}$ iff $\{(f_{\ast n_i})\mid i\in O\}$ is a permutation of $\{1\}\cup\{(f_{\ast n_i})\,|\,i\in E\}$, where $(f_{\ast n_i})$ is the eigenvalue of the map induced by $f$ in $H_{n_i}(X)=\mathbb{Q}$.
\end{theorem}
\begin{proof}
    From \cite{G-L}, we know that $L(f^m,X)=0$ for all $m\in\mathbb{N^+}$ iff
    \begin{equation*}
    \zeta_f(t)=1
\end{equation*}
But, by Equation~\eqref{ecuacion zeta lefschetz} this means
\begin{equation*}
    \frac{\prod_{i\in O}(1-t(f_{\ast i}))}{(1-t)\cdot\prod_{i\in E}(1-t(f_{\ast i}))}=1,
\end{equation*}
and the result follows immediately.
\end{proof}
We can also obtain some interesting results in the general case where $X$ consists of a sphere $S^{n_l}$ and $k_j$ disks of dimension $n_j$ for $j\in\{1,\ldots,l\}$ with $n_j>n_{j-1}$ (we admit $k_{l}=0$). The homology of $X$ is 
 \[ H_i(X) = \begin{cases} 
          \mathbb{Q}^{k_{l}+1} & \text{if}\;i=n_l, \\
          \mathbb{Q}^{k_{j}} & \text{if}\;i=n_j \;\text{and}\; j<l, \\
          \mathbb{Q} & \text{if}\, i=0,\\
          0 & \text{elsewhere}
       \end{cases}
    \]
In order to simplify the notation, we rename $k_{l}+1$ as $k_{l}$. Then the condition
\begin{equation*}
    \zeta_f(t)=\prod_{i=0}^{n}\mathrm{det}(\mathrm{Id}_i-tf_{\ast i})^{(-1)^{i+1}}=1
\end{equation*}
becomes
\begin{equation}\label{cociente para autovalores}
    \frac{\prod_{n_j\;\text{odd}}(-1)^{k_j}t^{k_j}p_j(\frac{1}{t})}{(1-t)\cdot\prod_{n_j\;\text{even}}(-1)^{k_j}t^{k_j}p_j(\frac{1}{t})}=\frac{\prod_{n_j\;\text{odd}}(-1)^{k_j}t^{k_j}p_j(\frac{1}{t})}{(-1)(t-1)\cdot\prod_{n_j\;\text{even}}(-1)^{k_j}t^{k_j}p_j(\frac{1}{t})}=1,
\end{equation}
where $p_j$ is the characteristic polynomial of the map induced by $f$ in $H_{n_j}$.
This tool appears in \cite{L}. Let us now call $RO$ the union (with multiplicity) of the non-zero roots (in $\mathbb{C}$) of the polynomials $p_j$ with $j$ such that $n_j$ is odd and let $RE$ be the union (with multiplicity) of the non-zero roots of the polynomials $p_j$ with $j$ such that $n_j$ is even. Given a polynomial $a_nx^n+\cdots+ a_ix^i$ with $n>\cdots>i$ (maybe $i=0$) we will call $a_i$ the \textit{last term} of the polynomial. Note that the last term of the polynomial $p_j(t)$ will be the leading coefficient of the polynomial $t^{k_j}p_j(\frac{1}{t})$.

The following theorem, translated to the language of bouquets of spheres (note that the homology of the bouquets and the bubble spaces is quite similar) is related to question (1) raised in \cite{S1} and to \cite[Theorem 5]{L-S2}.
\begin{theorem}\label{thm respuesta a pregunta abta}
    Let $X$ consists of a sphere $S^{n_l}$ and $k_j$ disks of dimension $n_j$ for $j\in\{1,\ldots,l\}$ with $n_j>n_{j-1}$. Then $L(f^m,X)=0$ for all $m\in\mathbb{Z}^+$ iff the following two conditions are satisfied:
    \begin{enumerate}
        \item The product of $(-1)^{\sum_{j\mid n_j\,\textbf{odd}}k_j}$ and the last terms of the polynomials $p_j$ with $j$ such that $n_j$ is odd equals the product $(-1)^{1+\sum_{j\mid n_j\,\textbf{even}}k_j}$ and the last terms of the polynomials $p_j$ with $j$ such that $n_j$ is even.
        \item $RO$ is a permutation of $\{1\}\cup RE$.
    \end{enumerate}
\end{theorem}
\begin{proof}
    From \eqref{cociente para autovalores}, it suffices to find sufficient and necessary conditions for 
    \begin{equation*}
        \prod_{n_j\;\text{odd}}(-1)^{k_j}t^{k_j}p_j(\frac{1}{t})=(-1)(t-1)\cdot\prod_{n_j\;\text{even}}(-1)^{k_j}t^{k_j}p_j(\frac{1}{t}).
    \end{equation*}
    But two polynomials are the same iff they have the same roots and the same leading coefficient. Now, since the roots of $t^{j_j}p_j(\frac{1}{t})$ are the inverse of the non-zero roots of $p_j$, the condition of having the same set (with multiplicity) of roots is equivalent to condition (2). Finally, since the leading coefficient of $t^{k_j}p_j(\frac{1}{t})$ is the last term of $p_j$, and the leading coefficient of $\prod_{n_j\;\text{odd}}(-1)^{k_j}t^{k_j}p_j(\frac{1}{t})$ is the product of the last terms of the $p_j$'s with $n_j$ odd together with the corresponding $(-1)^{k_j}$'s, whereas the leading coefficient of $(-1)(t-1)\cdot\prod_{n_j\;\text{even}}(-1)^{k_j}t^{k_j}p_j(\frac{1}{t})$ is the product of the last terms of the $p_j$'s with $n_j$ even together with the corresponding $(-1)^{k_j}$'s and $(-1)$, having the same leading coefficient is equivalent to condition (1).  
\end{proof}
\section{Some ideas with the relative Lefschetz number}\label{seccion relativo}
A second idea that we would like to highlight is the use of relative fixed point theory in these problems about periodicity. Given a map $f:(X,A)\rightarrow (X,A)$, with $A\subset X$, we define the Lefschetz number of the pair $(X,A)$ (or the relative Lefschetz number) in the usual sense considering maps induced in the relative (co)homology groups $H_\ast(X,A)$ (the only requirement is that these spaces are finite-dimensional and  vanish in almost every dimension). Indeed, if both $L(f,X)$ and $L(f,A)$ are defined, we have $L(f,X,A)=L(f,X)-L(f,A)$ \cite[Property 4.4]{Gorniewicz}. 


The relative Lefschetz number is important because of the following result, that also admits some generalizations in non-compact spaces (for these generalizations see, for example, \cite{Gorniewicz}):
\begin{theorem}\label{thm pto fijo relativo}
    Let $X$ be a compact simplicial complex, $A\subset X$ a subcomplex and $f:(X,A)\rightarrow (X,A)$ a continuous map. If $L(f,X,A)\neq 0$, then $f$ has a fixed point in $\overline{X\setminus A}$. 
\end{theorem}

So, this theorem opens the following question:
\begin{question}
    Given $f:(X,A)\rightarrow (X,A)$, are there necessary or sufficient conditions to guarantee that all the powers of $f$ have no fixed point in $\overline{X\setminus A}$? 
\end{question}
This fact has some relevance in dynamical systems, for example to study if a space is a repeller. As usual, a good point to begin with this problem is to see that the relative Lefschetz number vanishes for all the powers of the map.

\subsection{A first example: $\mathbb{Q}$-acyclic subspaces}
In this subsection we will use the Lefschetz number defined in cohomology (note that both the usual and the relative Lefschetz number agree in homology and cohomology since we are working with rational coefficients). From the equality $$L(f^m,X,A)=L(f^m,X)-L(f^m,A),$$ we see that $L(f^m,X,A)=0$ is equivalent to $L(f^m,X)=L(f^m,A)$. Now, if $A$ is $\mathbb{Q}$-acyclic (that means that all the positive homology groups vanish ---in particular if the space is contractible---), then $L(f^m,A)$ is always one. So, in this case, we must find whether $L(f^m,X)=1$ for all $m$.

As an application of this fact, we can generalize some results in \cite{G-K-N-S}.
\begin{theorem}
    Let $f:(\mathbb{C}P^n,A)\rightarrow (\mathbb{C}P^n,A)$ be a continuous map, where $A\subset \mathbb{C}P^n$ is a closed $\mathbb{Q}$-acyclic neighborhood. Then $L(f^m,\mathbb{C}P^n,A)=0$ for all $m\in\mathbb{Z}^+$ iff $a=0$, where $a$ is as in \cite[Equation 2.4]{G-K-N-S}, that is, the eigenvalue of $f^{\ast 2}$
\end{theorem}
\begin{proof}
    From \cite[Equation 2.4]{G-K-N-S} we have $L(f^m,\mathbb{C}P^n)=1+\sum_{i=1}^n(a^m)^i$. Consequently, $L(f^m,\mathbb{C}P^n,A)=0$ for all $m$ iff $\sum_{i=1}^n(a^m)^i=0$ for all $m$. But this only happens when $a=0$.
\end{proof}
We have an analogous theorem for $\mathbb{H}P^n$:
\begin{theorem}
        Let $f:(\mathbb{H}P^n,A)\rightarrow (\mathbb{H}P^n,A)$ be a continuous map, where $A\subset \mathbb{H}P^n$ is a closed $\mathbb{Q}$-acyclic neighborhood. Then $L(f^m,\mathbb{H}P^n,A)=0$ for all $m\in\mathbb{Z}^+$ iff $a=0$, where $a$ is the eigenvalue of $f^{\ast 4}$.
\end{theorem}
A similar reasoning can be done with the product of spheres. For example:
\begin{theorem}
    Let $f:(S^p\times S^q,A)\rightarrow (S^p\times S^q,A)$ be a continuous map, where $A\subset S^p\times S^q$ is a $\mathbb{Q}$-acyclic neighborhood and $p$ and $q$ are even, with $p\neq q$. In this case, $L(f^m,S^p\times S^q,A)=0$ iff $a=b=0$, where $a$ is the eigenvalue of $f^{\ast p}$ and $b$ is the eigenvalue of $f^{\ast q}$.
\end{theorem}
\begin{proof}
    From \cite[Case 3]{G-K-N-S}, we know that $$L(f^m,S^p\times S^q)=1+a^m+b^m+(ab)^m,$$ so $L(f^m,S^p\times S^q)=1$ for all $m$ iff $a=b=0$.
\end{proof}
The same idea works for the other cases in \cite{G-K-N-S}.
\subsection{The sphere and an equator ring.} Another example is the pair $(S^n,U)$, where $U$ is a ring neighborhood of the equator $S^{n-1}$ of $S^n$. All the relative homology groups vanish except $H_n(S^n,U)=\mathbb{Q}\oplus \mathbb{Q}$. Now, the equality
\begin{equation}\label{formula alternativa zeta}
    \zeta_f(t)=\prod_{k=0}^{n}\mathrm{det}(\mathrm{Id}_k-tf_{\ast k})^{(-1)^{k+1}}
\end{equation}
still holds with the relative Lefschetz number \cite[Proposition 5.13]{F}, where
\begin{equation*}
    \zeta_f(t):=\mathrm{exp}\biggl(\sum_{m\geq 1}\frac{L(f^m,S^n,U)}{m}t^m\biggr).
\end{equation*}
Thus, if we want to obtain $\zeta_f(t)=1$ (that is, $L(f^m)=0$ for all $m$), we must have 
\begin{equation}\label{eq esfera}
    \mathrm{det}\begin{pmatrix}
          1-ta&-tb\\
          -tc&1-td\\
        \end{pmatrix}=1,
\end{equation}
where 
\begin{equation*}
    \begin{pmatrix}
        a&b\\
        c&d\\
    \end{pmatrix}
\end{equation*}
is the matrix induced by $f$ in $H_n(S^n,U)$.
\begin{theorem}
    Let $f:(S^n,U)\rightarrow (S^n,U)$ be a map without periodic points in $\overline{S^n\setminus U}$. Then $a=-d$ and $bc=-a^2$, where
    \begin{equation*}
    \begin{pmatrix}
        a&b\\
        c&d\\
    \end{pmatrix}
\end{equation*}
is the matrix induced by $f$ in $H_n(S^n,U)$.
\end{theorem}
\begin{proof}
    From Equation~\eqref{eq esfera}, we must have 
    \begin{equation*}
        t^2(ad-bc)-t(a+d)=0, 
    \end{equation*}
    so the result follows.
\end{proof}
\subsection{The sphere and the solid torus.} The problem of the previous example is that the pair $(S^n,U)$ does not satisfy the usual hypotheses required to have the relative Wecken property \cite[Theorem 6.7]{Zhao} so we cannot even assure that $L(f^m,S^n,U)=0$ for all $m\in\mathbb{Z}^+$ will imply that $f$ has not fixed points in $\overline{S^n\setminus U}$.

This is the reason why we now present another example where the hypotheses in \cite[Theorem 6.7]{Zhao} are satisfied. For $n\geq 3$, let us consider the pair $(S^n, ST)$, where $ST$ is an inmersed solid torus $S^1\times D^{n-1}$ (we can visualize it in the case of the sphere $S^3$, that descomposes into two solid tori). The relative homology groups of this pair vanish except for $H_n(S^n,ST)=\mathbb{Q}$ and $H_2(S^n,ST)=\mathbb{Q}$.

We obtain the following theorem:

\begin{theorem}\label{thm ultimo teor}
    Let $f:(S^n,ST)\rightarrow (S^n,ST)$ be a continuous map without periodic points in $\overline{S^n\setminus ST}$. Then one of the following conditions holds:
    \begin{enumerate}
        \item If $n$ is odd, then the eigenvalue of $f_{\ast2}$ is the same as that of $f_{\ast n}$.
        \item If $n$ is even, then both $f_{\ast2}$ and $f_{\ast n}$ are zero.
    \end{enumerate}
\end{theorem}
\begin{proof}
   From Equation~\ref{formula alternativa zeta}, $f$ must satisfy
   \begin{equation*}
       \frac{(1-tf_{\ast n})^{(-1)^{n+1}}}{1-tf_{\ast 2}}=1.
   \end{equation*}
   The result follows then easily.
\end{proof}


\begin{thebibliography}{99}
    \bibitem{A-B} M. Arkowitz, R. F. Brown (2004). \emph{The Lefschetz-Hopf theorem and axioms for the Lefschetz number}, Fixed Point
Theory Appl., \textbf{1}, 1--11. 
    	\bibitem{Brown} R. F. Brown (1971). \emph{The Lefschetz Fixed Point Theorem}, Scott, Foresman and Company.
        \bibitem{F} J. Franks (1982). \emph{Homology and Dynamical Systems}, CBSM Regional Conf. Ser. in Math., \textbf{49}, Amer. Math. Soc.
        \bibitem{G-L} J. L. García Guirao, J. Llibre (2011). \emph{On the Lefschetz periodic point free continuous self-maps on connected
 compact manifolds}, Topology Appl., \textbf{158}, 2165–2169.
 \bibitem{G-L2} J. L. García Guirao, J. Llibre (2013), \emph{Periodic Structure of Transversal Maps on $\mathbb{C}P^n$, $\mathbb{H}P^n$ and $S^p\times S^q$}, Qual. Theory Dyn. Syst., \textbf{12}, (2), 417--425.
 \bibitem{G-L3} J. L. García Guirao, J. Llibre (2016), \emph{Periods of continuous maps on some compact spaces}, Houston J. Math., \textbf{42}, (3), 1047--1058. 
 \bibitem{G-S-U} M. J. González, V. F. Sirvent, R, Urzúa (2025). \emph{Homological data on the periodic structure of self-maps on wedge sums}, Qual. Theory Dyn. Syst., \textbf{24} (3), no. 127, 21 pp.
 \bibitem{Gorniewicz}
L.~G{\'o}rniewicz (2005), \emph{On the Lefschetz Fixed Point Theorem}, in: Handbook of Topological Fixed Point Theory, Springer, 43--82.
        \bibitem{G-K-N-S} G. Graff, A. Kaczkowska, P. Nowak-Przygodzki, J. Signerska (2012). \emph{Lefschetz periodic point free self-maps of compact manifolds}, Topology Appl., \textbf{159}, 2728--2735.
 \bibitem{H-R} M. Horecka, P. Ra\'zny (2022), \emph{A criterion for the existence of periodic points based on the eigenvalues of maps induced in cohomology}, Qual. Theory Dyn. Syst, \textbf{21}, (2), Art. 49, 12 pp.
 \bibitem{Jiang} B. J. Jiang (1983). \emph{Lectures on Nielsen Fixed Point Theory}, Contemporary Mathematics (14). American Mathematical Society.
 \bibitem{L} J. Llibre (2012). \emph{Periodic point free continuous self-maps on graphs and surfaces}, Topology Appl., \textbf{159}, 2228–2231.
 \bibitem{L-S1} J. Llibre, V. F. Sirvent (2013), \emph{Partially periodic point free self-maps on graphs, surfaces and other spaces}, J. Difference Equ. Appl., \textbf{19} (10), 1654--1662. 
 \bibitem{L-S2} J. Llibre, V. F. Sirvent (2018), \emph{On Lefschetz periodic point free self-maps}, J. Fixed Point Theory Appl.\textbf{20} (1), Art. 38 9 pp. 
        \bibitem{S2} V. F. Sirvent (2020). \emph{Partially Periodic Point Free Self-Maps on Product of Spheres and Lie Groups}, Qualitative Theory of Dynamical Systems, \textbf{19} (3) Art. 84, 12 pp.
        \bibitem{S1} V. F. Sirvent (2023). \emph{On partially periodic point free self-maps on the wedge sums of spheres}, J. Difference Equ. Appl., \textbf{29} (1), 19--31.
        
        \bibitem{Zhao}
X. Zhao (2005), \emph{Relative Nielsen Theory}, in: Handbook of Topological Fixed Point Theory, Springer, 659--684.
	\end{thebibliography}
\end{document}